\documentclass[final,1p,times]{elsarticle}

\usepackage[colorlinks=true]{hyperref}
\usepackage{amsfonts}
\usepackage{mathrsfs}
\usepackage{amsthm,amsmath,amssymb,mathrsfs,amsfonts,graphicx,graphics,latexsym,exscale,cmmib57,dsfont,bbm,amscd,ulem,curves}
\usepackage{color,soul}
\usepackage{makeidx}
\makeindex
\newtheorem{thm}{Theorem}[section]
\newtheorem{lem}[thm]{Lemma}
\newtheorem{cor}[thm]{Corollary}

\theoremstyle{definition}
\newtheorem{defn}[thm]{Definition}
\newtheorem{exa}[thm]{Example}
\newtheorem*{note}{Note}
\journal{xxx}

\begin{document}

\begin{frontmatter}
\title{Uniform positive recursion frequency of any minimal dynamical system on a compact space}

\author{Xiongping Dai}
\ead{xpdai@nju.edu.cn}
\address{Department of Mathematics, Nanjing University, Nanjing 210093, People's Republic of China}

\begin{abstract}
Using Gottschalk's notion\,---\,weakly locally almost periodic point, we show in this paper that if $f\colon X\rightarrow X$ is a minimal continuous transformation of a compact Hausdorff space $X$ to itself, then for all entourage $\varepsilon$ of $X$,
\begin{equation*}
\inf_{x\in X}\left\{\liminf_{N-M\to\infty}\frac{1}{N-M}\sum_{n=M}^{N-1}1_{\varepsilon[x]}(f^nx)\right\}>0.
\end{equation*}
An analogous assertion also holds for each minimal $C^0$-semiflow $\pi\colon \mathbb{R}_+\times X\rightarrow X$ and for any minimal transformation group with discrete amenable phase group.
\end{abstract}

\begin{keyword}
Minimality; weakly locally almost periodic point; semiflow.

\medskip
\MSC[2010] 37B05 $\cdot$ 20M20
\end{keyword}
\end{frontmatter}

\section{Introduction}\label{sec1}
Let $f\colon X\rightarrow X$ be any continuous transformation of a compact $T_2$-space $X$. We say that $(f,X)$ is a \textit{minimal} cascade~\cite{NS, Fur, Aus} if the orbit $\{f^nx\,|\,n=0,1,2,\dotsc\}$
is dense in $X$ for all $x\in X$.

Since here $X$ is compact $T_2$, there is a unique compatible symmetric uniform structure $\mathscr{U}_X$ on $X$ (cf.~\cite[Appendix~II]{Aus}). For any $\varepsilon\in\mathscr{U}_X$ write $\varepsilon[x]=\{y\in X\,|\,(x,y)\in\varepsilon\}$ which is an open neighborhood of $x$. Given $\varepsilon\in\mathscr{U}_X$, we define
$$
N_f(x,\varepsilon[y])=\left\{n\in\mathbb{Z}_+\,|\,f^nx\in\varepsilon[y]\right\}
$$
for all $x,y\in X$. Let $|\centerdot|$ be the counting measure on $\mathbb{Z}_+$ and then we can define the ``Banach lower density'' as follows:
$$
\textmd{BD}_*(f,x,\varepsilon)=\liminf_{N-M\to\infty}\frac{\big{|}[M,N)\cap N_f(x,\varepsilon[x])\big{|}}{N-M}\quad \forall x\in X,\ \varepsilon\in\mathscr{U}_X.
$$
Given any $\varepsilon\in\mathscr{U}_X$ we shall consider in this short paper whether $\textmd{BD}_*(f,x,\varepsilon)>0$ uniformly for $x\in X$ or not.

If $(f,X)$ is minimal and if $\mu$ is an ergodic Borel probability measure of $(f,X)$, then by the classical Birkhoff ergodic theorem~\cite{NS} it follows that for all $y\in X$ and $\varepsilon\in\mathscr{U}_X$,
\begin{equation*}
\lim_{N\to\infty}\frac{1}{N}\big{|}[0,N)\cap N_f(x,\varepsilon[y])\big{|}=\mu(\varepsilon[y])>0\quad \textit{for }\mu\textit{-a.e. }x\in X.
\end{equation*}
However, this ``$\mu$-a.e.'' does not imply the desired conclusion that $\textmd{BD}_*(f,x,\varepsilon)>0$ uniformly for $x\in X$.

In this short paper, by using Gottschalk's notion\,---\,``weakly locally almost periodic'' point of transformation semigroup (cf.~Definition~\ref{def3} in $\S\ref{sec2}$), we will be able to show that if $(f,X)$ is a minimal cascade, then for all $\varepsilon\in\mathscr{U}_X$ there follows $\textmd{BD}_*(f,x,\varepsilon)>0$ uniformly for $x\in X$;
see Theorem~\ref{thm15} proved in $\S\ref{sec3}$. Moreover, this assertion also holds for all minimal continuous-time $C^0$-semiflow; see Theorem~\ref{thm17} below. Finally in $\S\ref{sec4}$, we will consider this question for any minimal flow with discrete amenable phase group; see Theorem~\ref{thm4.4}.

\section{Weakly locally almost periodic points of any semiflow}\label{sec2}
Let $X$ be a compact $T_2$-space and by $\mathfrak{N}_x$ we will mean the filter of neighborhoods of $x$ in $X$. Let $T$ be a multiplicative topological semigroup with an identity $e$ and by $\mathscr{K}_T$ it means the collection of non-empty compact subsets of $T$. When $T$ is discrete, each $K\in\mathscr{K}_T$ is finite.

As usual, a \textit{semiflow} or \textit{transformation semigroup} is understood as a pair $(T,X)$ where $T$ is called the phase semigroup and $X$ the phase space such that $(t,x)\mapsto tx$ of $T\times X$ to $X$ is jointly continuous with $ex=x$ and $t(sx)=(ts)x$ for all $x\in X$ and $s,t\in T$.

In $\S\ref{sec3}$ we will be mainly interested to the classical cases of $T=(\mathbb{Z}_+,+)$ and $T=(\mathbb{R}_+,+)$ with the usual topologies, and in $\S\ref{sec4}$ we will consider transformation groups with discrete amenable phase groups. 

First of all we need to introduce in this section some basic notions and results for our later discussion.

\begin{defn}[{cf.~\cite{Fur,AD}}]\label{def1}
\begin{enumerate}
\item A subset $A$ of $T$ is called \textit{syndetic} in $T$ if there exists a $K\in\mathscr{K}_T$ such that $Kt\cap A\not=\emptyset$ for all $t\in T$.

\item An $x\in X$ is referred to as an \textit{almost periodic} (a.p.) \textit{point} of $(T,X)$ if
$$
N_T(x,U)=\{t\in T\,|\,tx\in U\}
$$
is syndetic in $T$. Equivalently, $x\in X$ is an a.p. point if and only if $\overline{Tx}$ is a minimal set of $(T,X)$.
\end{enumerate}
\end{defn}

Since $T$ does not need to be abelian here, `$Kt\cap A\not=\emptyset$' is not permitted to be replaced by `$tK\cap A\not=\emptyset$' in Definition~\ref{def1}.

\begin{lem}\label{lem2}
If $A$ is a syndetic subset of $\mathbb{Z}_+$ with $K=\{0,1,\dotsc,L\}$ such that $(K+t)\cap A\not=\emptyset$ for all $t\in\mathbb{Z}_+$, then $A$ has positive lower density $\liminf_{N\to\infty}\frac{1}{N}|[0,N)\cap A|\ge\frac{1}{L+1}$.
\end{lem}

\begin{defn}[{Gottschalk~\cite{G56} for $T$ in groups}]\label{def3}
Let $(T,X)$ be any semiflow and let $x\in X$. Then:
\begin{enumerate}
\item[(1)] $x$ is called a \textit{weakly locally almost periodic} (w.l.a.p.) point of $(T,X)$ if given $U\in\mathfrak{N}_x$ there exist a $V\in\mathfrak{N}_x$, a syndetic subset $A$ of $T$ and a $K\in\mathscr{K}_T$ such that for all $x^\prime\in V$ there is a subset $B$ of $T$ such that
\begin{gather*}
Kt\cap B\not=\emptyset\ \forall t\in A\quad \textrm{and}\quad Bx^\prime\subseteq U.
\end{gather*}
$(T,X)$ is called \textit{w.l.a.p.} if it is weakly locally almost periodic at every point $x$ of $X$.

\item[(2)] If here $B$ is just independent of $x^\prime\in V$ in (1), then $x$ is called a \textit{locally almost periodic} (l.a.p.) point of $(T,X)$; that is, for each $U\in\mathfrak{N}_x$ there exist a $V\in\mathfrak{N}_x$ and a syndetic set $B$ in $T$ such that $BV\subseteq U$.
\end{enumerate}
\end{defn}

Clearly, any l.a.p. point is a w.l.a.p. point; but the w.l.a.p. is actually weaker than the l.a.p., see Theorem~\ref{thm2.6} and Example~\ref{exa14} below.

There are two simple equivalent descriptions of w.l.a.p. point of any semiflow as follows, which are very useful for our main results.

\begin{lem}\label{lem4}
Let $(T,X)$ be any semiflow and let $x\in X$. Then the following three conditions are pairwise equivalent.
\begin{enumerate}
\item[$(1)$] $x$ is a w.l.a.p. point of $(T,X)$.

\item[$(2)$] Given $U\in\mathfrak{N}_x$ there exist a $V\in\mathfrak{N}_x$ and a $K\in\mathscr{K}_T$ such that $TV\subseteq K^{-1}U$.

\item[$(3)$] For each $U\in\mathfrak{N}_x$, there exist a $V\in\mathfrak{N}_x$ and a $K\in\mathscr{K}_T$ such that if $x^\prime\in V$ there is a subset $B$ of $T$ with $Kt\cap B\not=\emptyset\ \forall t\in T$ and $Bx^\prime\subseteq U$.
\end{enumerate}
Here $K^{-1}U=\bigcup_{k\in K}k^{-1}U$.
\end{lem}

\begin{proof}
$(1)\Rightarrow(2)$. Let $x$ be a w.l.a.p. point of $(T,X)$ and $U\in\mathfrak{N}_x$. Then there exist a $V\in\mathfrak{N}_x$, a syndetic subset $A$ of $T$, and a $C\in\mathscr{K}_T$ such that $y\in V$ implies that there is a subset $B$ of $T$ with $A\subseteq C^{-1}B$ and $By\subseteq U$. Take an $F\in\mathscr{K}_T$ with $T=F^{-1}A$ for $A$ is syndetic in $T$, and set $K=CF$ that is independent of $y\in V$. Hence $T=F^{-1}C^{-1}B=K^{-1}B$ and $Ty\subseteq K^{-1}U$ so that $TV\subseteq K^{-1}U$.

$(2)\Rightarrow(3)$. Given any $x^\prime\in V$, define a set $T_k=\{t\in T\,|\,tx^\prime\in k^{-1}U\}$ for each $k\in K$. Then $T=\bigcup_{k\in K}T_k$ and define $B=\bigcup_{k\in K}kT_k$. Then $Bx^\prime\subseteq U$ and $Kt\cap B\not=\emptyset$ for all $t\in T$. This thus concludes condition (3).

$(3)\Rightarrow(1)$. This is evident for $T$ itself is syndetic in $T$.

Therefore the proof of Lemma~\ref{lem4} is completed.
\end{proof}

It is known that when $T$ is in groups, l.a.p. is independent of the topology of $T$ (cf., e.g.,~\cite{MW, AM}). In fact, w.l.a.p. also has this property in groups as follows:

\begin{cor}
Let $(T,X)$ be a flow with phase group $T$ which is not discrete and $x\in X$. Then $x$ is a w.l.a.p. point of $(T,X)$ if and only if $x$ is a w.l.a.p. point of $(T,X)$ with $T$ to be discrete.
\end{cor}

\begin{proof}
Since any finite subset of $T$ must be a compact subset of $T$, hence the sufficiency holds trivially. Conversely, let $x$ is w.l.a.p. for $(T,X)$; then we shall show that $x$ is a w.l.a.p. for $(T,X)$ with $T$ discrete. In fact, given any open $U\in\mathfrak{N}_x$ let $U^\prime$ be a closed neighborhood of $x$ with $U^\prime\subset U$. By (2) of Lemma~\ref{lem4} there exist a $V\in\mathfrak{N}_x$ and a compact subset $K$ of $T$ such that $TV\subseteq K^{-1}U^\prime$. Since $T\times X\rightarrow X$ is jointly continuous and $T$ is a topological group, there is a neighborhood $L$ of $e$ in $T$ such that $(Lk)^{-1}U^\prime\subseteq k^{-1}U$ for all $k\in K$. Then by the compactness of $K$, we can pick a finite set $F\subset K$ such that $K^{-1}U^\prime\subseteq F^{-1}U$ so that $TV\subseteq F^{-1}U$. By Lemma~\ref{lem4} again, $x$ is an l.w.a.p. point of $(T,X)$ with $T$ discrete.
\end{proof}

By (1) of Definition~\ref{def3}, any w.l.a.p. point is an a.p. point. In fact, an a.p. point must be w.l.a.p. restricted to its orbit closure as shown by Theorem~\ref{thm2.6} below. Theorem~\ref{thm2.6} is the important tool for our main results of this paper, which meanwhile shows that the w.l.a.p. is essentially weaker than the l.a.p. property.

\begin{thm}\label{thm2.6}
Let $(T,X)$ be a minimal semiflow; then each point of $X$ is w.l.a.p. for $(T,X)$.
\end{thm}

\begin{proof}
Let $x\in X$ and $U\in\mathfrak{N}_x$ an open neighborhood of $x$. Since $\{t^{-1}U\,|\,t\in T\}$ is an open cover of $X$ and $X$ is compact, there is a finite set $K\subseteq T$ such that $K^{-1}U=X$. Then $TV\subseteq K^{-1}U$ for all $V\in\mathfrak{N}_x$. Thus $x$ is a w.l.a.p. point of $(T,X)$ by (2) of Lemma~\ref{lem4}.
\end{proof}

Since a minimal flow/semiflow is in general not l.a.p. (cf.~\cite{AM}), hence the w.l.a.p. is weaker than the l.a.p. dynamics by Theorem~\ref{thm2.6}.

\begin{defn}\label{def6}
\begin{enumerate}
\item An $x\in X$ is called an \textit{accessible} point of $(T,X)$, or we say $(T,X)$ is \textit{accessible at $x$}, if there exist points $y,z\in X, y\not=z$ and a net $\{t_n\}$ in $T$ such that $t_n(y,z)\to(x,x)$.
\item $(T,X)$ is called a \textit{distal} semiflow if given $x,y\in X$ with $x\not=y$, there is an $\alpha\in\mathscr{U}_X$ such that $(tx,ty)\not\in\alpha$.
\end{enumerate}
\end{defn}

Clearly $(T,X)$ is distal if and only if there exist no accessible points of $(T,X)$. If a minimal semiflow has an accessible point, then each point is accessible.

\begin{defn}
\begin{enumerate}
\item We say $(T,X)$ is \textit{equicontinuous at a point} $x\in X$ in case given $\varepsilon\in\mathscr{U}_X$ there is a $\delta\in\mathscr{U}_X$ such that if $y\in\delta[x]$ then $(tx,ty)\in\varepsilon$ for all $t\in X$.
\item When $(T,X)$ is equicontinuous at every point of $X$, $(T,X)$ is called \textit{equicontinuous}.
\end{enumerate}
Then by Lebesque's covering lemma, $(T,X)$ is equicontinuous if and only if given $\varepsilon\in\mathscr{U}_X$ there is a $\delta\in\mathscr{U}_X$ such that $T\delta\subseteq\varepsilon$ (cf.~\cite[Lemma~1.6]{AD}).
\end{defn}

The following lemma is an important tool, which is motivated by \cite[Lemma~1]{G56}.

\begin{lem}\label{lem8}
Let $(T,X)$ be a semiflow with $T$ a discrete semigroup. Suppose that
there exists $x\in X$ such that $(T,X\times X)$ is w.l.a.p. at $(x,x)$ but $(T,X)$ is not equicontinuous at $x$. Then $x$ is an accessible point of $(T,X)$ and hence $(T,X)$ is not distal.
\end{lem}

\begin{proof}
Since $(T,X)$ is not equicontinuous at $x$, there is an $\varepsilon\in\mathscr{U}_X$ with $T(N\times N)\not\subset\varepsilon$ and then $T(N\times N)\cap\varepsilon^\prime\not=\emptyset$ for every $N\in\mathfrak{N}_x$, where $\varepsilon^\prime$ is the complement of $\varepsilon$ in $X\times X$. Define two families of subsets of $X\times X$ as follows:
\begin{gather*}
\mathcal{F}_+=\{T(N\times N)\cap\varepsilon^\prime\,|\,N\in\mathfrak{N}_x\}\quad\textrm{and}\quad \mathcal{F}_-=\{T^{-1}(N\times N)\cap(\varepsilon^\prime\cap X\times X)\,|\,N\in\mathfrak{N}_x\}.
\end{gather*}
Next we will show that $\bigcap\mathcal{F}_-\not=\emptyset$.

Indeed, since $T(N\times N)\cap\varepsilon^\prime\not=\emptyset$ for all $N\in\mathfrak{N}_x$ and $\bar{\mathcal{F}}_+=\{\textrm{cls}_{X\times X}F\,|\,F\in\mathcal{F}_+\}$ is a closed filter-base in the compact set $\varepsilon^\prime$, then $\bigcap\bar{\mathcal{F}}_+\not=\emptyset$. Let $N\in\mathfrak{N}_x$ be any \textit{closed} neighborhood of $x$. Since $(T,X\times X)$ is w.l.a.p. at $(x,x)$ by hypothesis, there exists an $M\in\mathfrak{N}_x$ (so that $V=M\times M\in\mathfrak{N}_{(x,x)}$ in $X\times X$) and a finite subset $K$ of $T$ such that $T(M\times M)\subset K^{-1}(N\times N)$ by (2) of Lemma~\ref{lem4} for $U=N\times N$ whence $\textrm{cls}_{X\times X}{T(M\times M)}\subset X\times X\cap K^{-1}(N\times N)\subseteq X\times X\cap T^{-1}(N\times N)$ and $\textrm{cls}_{X\times X}({T(M\times M)\cap\varepsilon^\prime})\subseteq T^{-1}(N\times N)\cap(\varepsilon^\prime\cap X\times X)$. Therefore, $\bigcap\mathcal{F}_-\supseteq\bigcap\bar{\mathcal{F}}_+\not=\emptyset$.

Now let $(w,z)\in\bigcap\mathcal{F}_-$. Then $(w,z)\not\in\varepsilon$ (so $w\not=z$) and $(w,z)\in X\times X\cap T^{-1}(N\times N)$ for all $N\in\mathfrak{N}_x$. This implies that for each $\delta\in\mathscr{U}_X$ there exists $t_\delta\in T$ such that $(t_\delta w,t_\delta z)\in\delta[x]\times\delta[x]$ and thus $(w,z)$ is a proximal pair approaching to $x$ for $(T,X)$.

Therefore we have concluded that $x$ is an accessible point of $(T,X)$.
\end{proof}

\begin{defn}[\cite{DX,AD}]
$(T,X)$ is \textit{uniformly almost periodic} (u.a.p.) if and only if given $\varepsilon\in\mathscr{U}_X$ there is a syndetic set $A$ in $T$ such that $Ax\subseteq\varepsilon[x]$ for all $x\in X$.
\end{defn}

Recall that a semiflow $(T,X)$ is said to be \textit{surjective} if each $t\in T$ is a surjective self-map of $X$. Every minimal semiflow with amenable phase semigroup must be surjective (cf.~\cite{AD}).

Let $(T,X)$ be any semiflow with phase semigroup $T$ in the following three lemmas.

\begin{lem}[\cite{DX, AD}]\label{lem10}
$(T,X)$ is u.a.p. if and only if it is equicontinuous surjective.
\end{lem}

\begin{lem}[\cite{AD}]\label{lem11}
If $(T,X)$ is equicontinuous surjective, then it is distal.
\end{lem}

\begin{lem}[\cite{Dai}]\label{lem12}
$(T,X)$ is u.a.p. if and only if $(T,X)$ is l.a.p. and distal.
\end{lem}

\begin{thm}[{cf.~\cite[Theorem~1]{G56} for $T$ in groups}]\label{thm13}
Let $(T,X)$ be a semifflow with discrete phase semigroup $T$. Then the following statements are pairwise equivalent.
\begin{enumerate}
\item[$(1)$] $(T,X)$ is u.a.p.;
\item[$(2)$] $(T,X)$ is equicontinuous surjective;
\item[$(3)$] $(T,X)$ is l.a.p. distal;
\item[$(4)$] $(T, X\times X)$ is w.l.a.p. at every point of the diagonal $\varDelta$ of $X\times X$ and $(T,X)$ is distal.
\end{enumerate}
\begin{note}
In fact, $(1)\Leftrightarrow(2)\Leftrightarrow(3)$ is independent of the choice of the topology of $T$.
\end{note}
\end{thm}

\begin{proof}
First of all, $(1)\Leftrightarrow(2)\Leftrightarrow(3)$ follows from Lemmas~\ref{lem10}, \ref{lem11} and \ref{lem12}; and $(3)\Rightarrow(4)$ is obvious.
Now let $(4)$ hold. Then Lemma~\ref{lem8} follows that $(T,X)$ is equicontinuous. Since $(T,X)$ is distal, so it is surjective and thus (2) holds.
The proof of Theorem~\ref{thm13} is therefore completed.
\end{proof}

The l.a.p. implies the diagonalwise product semiflow is w.l.a.p. at the diagonal. However, the following Example~\ref{exa14} shows that the w.l.a.p. does not imply this.

\begin{exa}\label{exa14}
Let $\alpha>0$ be an irrational number sufficiently small and let $\langle \gamma_n\rangle_{n=1}^\infty$ be a sequence of rational numbers with $\gamma_n\to 1$ as $n\to\infty$. Let $X$ be the compact subset of the plane defined by
\begin{gather*}
X=\textrm{cls}_{\mathbb{R}\times \mathbb{R}}{\left\{(r,\theta)\,|\,\theta\in\mathbb{R},\ r=(1+\alpha+{\gamma_n}n^{-1})\pi,\ n=1,2,\dotsc\right\}}
\end{gather*}
where $(r,\theta)$ are polar coordinates of the plane.
We consider the homeomorphism $\tau$ of $X$ onto itself given by $(r,\theta)\mapsto\tau(r,\theta)=(r,\theta+r)$ and define $T=\{\tau^n\,|\,n\in\mathbb{Z}_+\}$. Clearly \begin{itemize}
\item \textit{$T$ acts distally on $X$; i.e., $(T,X)$ is distal.}
\end{itemize}
To see that $T$ is not equicontinuous acting on $X$, let $r_0=(1+\alpha)\pi$ and $r_n=(1+\alpha+\gamma_nn^{-1})\pi$ for each $n\ge1$; then $(r_n,\theta)\to(r_0,\theta)$ as $n\to\infty$. But
\begin{gather*}
\tau^{2n}(r_0,\theta)=(r_0,\theta+2n(1+\alpha)\pi)\quad\textrm{and}\quad\tau^{2n}(r_{2n},\theta)=(r_{2n},\theta+2n(1+\alpha)\pi+\gamma_{2n}\pi)
\end{gather*}
conclude that $T$ is not equicontinuous at any point $(r_0,\theta)$. Thus \textit{$T$ is not equicontinuous on $X$}. Further by Theorem~\ref{thm13}, it follows that \begin{itemize}
\item \textit{$(T,X)$ is not an l.a.p. semiflow.}
\end{itemize}
However,
\begin{itemize}
\item \textit{$(T,X)$ is a w.l.a.p. semiflow.}
\end{itemize}

Indeed, it is easy to see that $(T,X)$ is l.a.p. at every point $(r,\theta)\in X$ with $r\not=r_0$. Now given $(r_0,\theta)\in X$ and $\varepsilon>0$, since $1+\alpha$ is irrational, we can find a finite subset $K$ of $T$ such that
\begin{gather*}
K^{-1}(\{(r_0,\varphi)\,|\,\theta-\varepsilon<\varphi<\theta+\varepsilon\})=\{(r_0,\varphi)\,|\,0\le\varphi\le2\pi\}.
\end{gather*}
Since $r_n\to r_0$ as $n\to\infty$, we can obtain
\begin{gather*}
K^{-1}(\{(r_n,\varphi)\,|\,\theta-\varepsilon<\varphi<\theta+\varepsilon\})=\{(r_n,\varphi)\,|\,0\le\varphi\le2\pi\}.
\end{gather*}
as $n$ sufficiently large. Then by Lemma~\ref{lem4}, it follows that $T$ is w.l.a.p. at the point $(r_0,\theta)$. This proves the conclusion.
\end{exa}

\section{Uniform recursion frequency of minimal semiflows}\label{sec3}
Let $X$ be a compact $T_2$-space with the uniform structure $\mathscr{U}_X$. Given $\varepsilon\in\mathscr{U}_X$, $\frac{\varepsilon}{3}$ will stand for an entourage in $\mathscr{U}_X$ such that $\frac{\varepsilon}{3}\circ\frac{\varepsilon}{3}\circ\frac{\varepsilon}{3}\subseteq\varepsilon$.

In this section we will be concerned with the classical semiflows $(T,X)$ with phase semigroup $\mathbb{Z}_+$ or $\mathbb{R}_+$.
If $f\colon X\rightarrow X$ is a continuous transformation of $X$, then $(f,X)$ is called a cascade. It induces a semiflow $\mathbb{Z}_+\times X\rightarrow X$ by $(t,x)\mapsto f^tx$.

\begin{thm}\label{thm15}
Let $(f,X)$ be a minimal cascade; then
$\inf_{x\in X}\mathrm{BD}_*(f,x,\varepsilon)>0$
for all $\varepsilon\in\mathscr{U}_X$.
\end{thm}

\begin{proof}
Let $\varepsilon\in\mathscr{U}_X$ be any given; and then we can take some $\delta\in\mathscr{U}_X$ so small that $\frac{\varepsilon}{3}[x]\subseteq\varepsilon[y]$ for all $x\in X$ and any $y\in\delta[x]$. First by Theorem~\ref{thm2.6}, $(f,X)$ is w.l.a.p. so that by (3) of Lemma~\ref{lem4}, it follows that for each $x\in X$ there exist
a $\delta_x\in\mathscr{U}_X$ with $\delta_x\le\delta$ and an integer $K_x\ge1$ such that for each $y\in\delta_x[x]$ there is a subset $B$ of $\mathbb{Z}_+$ with the properties:
$$
\left\{f^ty\,|\,t\in B\right\}\subseteq\frac{\varepsilon}{3}[x]\quad \textrm{and}\quad \left[t,K_x+t\right]\cap B\not=\emptyset\ \forall t\in\mathbb{Z}_+.
$$
Since $X$ is compact, there is a finite subset $\{x_1,\dotsc,x_n\}$ of $X$ such that $X=\delta_{x_1}[x_1]\cup\dotsm\cup\delta_{x_n}[x_n]$.
Therefore, for all $x\in X$, there are some $1\le i\le n$ with $x\in\delta_{x_i}[x_i]$ and some $B_{i,x}\subseteq \mathbb{Z}_+$ such that
$$
\{f^tx\,|\,t\in B_{i,x}\}\subseteq\frac{\varepsilon}{3}[x_i]\subseteq\varepsilon[x]\quad \textrm{and}\quad [t,K_{x_i}+t]\cap B_{i,x}\not=\emptyset\ \forall t\in\mathbb{Z}_+.
$$
Let $K=\max\{K_{x_i}\,|\,1\le i\le n\}$, which is independent of $x$. Then
\begin{equation*}
\liminf_{N-M\to\infty}\frac{\big{|}[M,N)\cap N_f(x,\varepsilon[x])\big{|}}{N-M}\ge\frac{1}{K+1}.
\end{equation*}
Since $\varepsilon\in\mathscr{U}_X$ and $x\in X$ both are arbitrary, this proves Theorem~\ref{thm15}.
\end{proof}

Let $\pi\colon\mathbb{R}_+\times X\rightarrow X,\ (t,x)\mapsto t\cdot x$ be a classical $C^0$-semiflow; that is, $\pi(t,x)$ is jointly continuous such that $0\cdot x=x$ and $(s+t)\cdot x=s\cdot(t\cdot x)$ for all $x\in X$ and $s,t\in\mathbb{R}_+$, where we have identified each $t\in\mathbb{R}_+$ with the transition $\pi(t,\centerdot)\colon X\rightarrow X$.

To prove the continuous-time version of Theorem~\ref{thm15}, we will need the following technical lemma.

\begin{lem}\label{lem16}
Let $\pi\colon\mathbb{R}_+\times X\rightarrow X$ be a classical $C^0$-semiflow and $\varepsilon\in\mathscr{U}_X$. Then there exist $\delta\in\mathscr{U}_X$ and $\alpha>0$ such that for any $x,y\in X$ and $t\in\mathbb{R}_+$, if $t\cdot y\in\delta[x]$ then $[t,t+\alpha]\cdot y\subseteq\varepsilon[x]$.
\end{lem}

\begin{proof}
Otherwise, there are $x\in X$ and nets $y_n\to x$ and $\alpha_n\to0$ such that $\alpha_n\cdot y_n\not\in\varepsilon[x]$. But $\pi(\alpha_n,y_n)=\alpha_n\cdot y_n\to x$, a contradiction.
\end{proof}

Let $1_A(x)=1$ if $x\in A$ and $=0$ if $x\not\in A$ be the indicator function of a set $A\subseteq X$. Then by Lemma~\ref{lem16}, a light modification of the proof of Theorem~\ref{thm15} follows the following theorem:

\begin{thm}\label{thm17}
If $\pi\colon\mathbb{R}_+\times X\rightarrow X,\ (t,x)\mapsto t\cdot x$ is a minimal $C^0$-semiflow, then
\begin{equation*}
\inf_{x\in X}\liminf_{T_2-T_1\to\infty}\frac{1}{T_2-T_1}\int_{T_1}^{T_2}1_{\varepsilon[x]}(t\cdot x)dt>0
\end{equation*}
for all $\varepsilon\in\mathscr{U}_X$.
\end{thm}

It should be noticed that if $\textmd{BD}_*(f,x,\varepsilon)>0$ for all $\varepsilon\in\mathscr{U}_X$, then $x$ is an a.p. point of $(f,X)$. This fact also holds for amenable group actions; see Theorem~\ref{thm4.5} in $\S\ref{sec4}$.

\section{Minimal flows with discrete amenable phase groups}\label{sec4}
By a \textit{discrete amenable group}, it always refers to as a discrete group $T$ with a Haar measure $|\centerdot|$ (in fact the counting measure), for which it holds the \textit{left F{\o}lner condition}:
\begin{itemize}
\item Given $K\in\mathscr{K}_T$ and $\varepsilon>0$ there exists an $F\in\mathscr{K}_T$ such that $|KF\vartriangle F|<\varepsilon|F|$.
\end{itemize}
Here $A\vartriangle B=(A\setminus B)\cup(B\setminus A)$ is the symmetric difference of sets $A,B$ in $T$. Any abelian group is amenable.

We will need two lemmas.

\begin{lem}[{cf.~\cite[Theorem~4.16]{Pat}}]\label{lem4.1}
If $T$ is a discrete amenable group, then there exists a (left) F{\o}lner net $\{F_n\,|\,n\in D\}$ in $T$; that is, $\{F_n\,|\,n\in D\}$ is a net in $\mathscr{K}_T$ such that $\lim_n|tF_n\vartriangle F_n|/|F_n|=0$ for all $t\in T$.
\end{lem}

The following technical lemma is essentially contained in the proof of Hindman and Strauss~\cite[Theorem~4.11]{HS}.

\begin{lem}\label{lem4.2}
If $\{F_n\,|\,n\in D\}$ is a F{\o}lner net in a discrete amenable group $T$, then it satisfies the condition:
\begin{itemize}
\item[$(*)$] Given any $H\in\mathscr{K}_T$ there exists an index $m_H\in D$ such that $|F_n|\le 2\big{|}\bigcap_{h\in H}h^{-1}F_n\big{|}$ for all $n\ge m_H$.
\end{itemize}
\end{lem}

\begin{proof}
Let $H\in\mathscr{K}_T$ be any given. Since
$\lim_{n\in D}\frac{|tF_n\vartriangle F_n|}{|F_n|}=0$ for all $t\in T$ and $H$ is finite, there is some index $m_H\in D$ such that
$$
|F_n\setminus a^{-1}F_n|\le\frac{1}{2|H|}\cdot|F_n| \quad \forall n\ge m_H\textrm{ and }a\in H.
$$
Then for all $n\ge m_H$,
\begin{equation*}\begin{split}
|F_n|&=\left|F_n\cap{\bigcap}_{a\in H}a^{-1}F_n\right|+\left|F_n\setminus{\bigcap}_{a\in H}a^{-1}F_n\right|\\
&=\left|F_n\cap{\bigcap}_{a\in H}a^{-1}F_n\right|+\left|{\bigcup}_{a\in H}(F_n\setminus a^{-1}F_n)\right|\\
&\le\left|F_n\cap{\bigcap}_{a\in H}a^{-1}F_n\right|+{\sum}_{a\in H}\left|F_n\setminus a^{-1}F_n\right|\\
&\le\left|F_n\cap{\bigcap}_{a\in H}a^{-1}F_n\right|+\frac{1}{2}\left|F_n\right|.
\end{split}\end{equation*}
Thus, $\frac{1}{2}\left|F_n\right|\le\left|F_n\cap{\bigcap}_{a\in H}a^{-1}F_n\right|\le\left|{\bigcap}_{a\in H}a^{-1}F_n\right|$ as desired.
\end{proof}

\begin{lem}\label{lem4.3}
Let $T$ be a discrete amenable group. If $B$ is a syndetic subset of $T$ such that there exists a $K\in\mathscr{K}_T$ with $Kt\cap B\not=\emptyset$ for all $t\in T$, then $\liminf_{n}\frac{|B\cap F_n|}{|F_n|}\ge\frac{1}{2|K|}$ for all F{\o}lner net $\{F_n\,|\,n\in D\}$ in $T$.
\end{lem}

\begin{proof}
First by Lemma~\ref{lem4.2} for $H=K$ and $\{F_n\,|\,n\in D\}$, we can pick $m_K\in D$ such that
$$
|F_n|\le 2\left|{\bigcap}_{h\in K}h^{-1}F_n\right|\quad \forall n\ge m_K.
$$
Next for all $n\ge m_K$ in $D$ we shall show that $|B\cap F_n|\ge\frac{1}{2|K|}\cdot|F_n|$ and thus
$$
\liminf_{n\in D}\frac{|B\cap F_n|}{|F_n|}\ge\frac{1}{2|K|}.
$$
For this, we define a function
$\tau\colon{\bigcap}_{t\in K}t^{-1}F_n\rightarrow(B\cap F_n)\times K$
as follows: given $g\in{\bigcap}_{t\in K}t^{-1}F_n$, we take $k\in K$ such that $kg\in B$. Since $g\in t^{-1}F_n$ for all $t\in K$, hence $kg\in F_n$. Now we set $\tau(g)=(kg, k)\in(B\cap F_n)\times K$. Since $T$ is a group, we can see that for any $g,g^\prime\in{\bigcap}_{t\in K}t^{-1}F_n$,
$$
(kg,k)=(k^\prime g^\prime,k^\prime)\Rightarrow g=g^\prime.
$$
Thus $\tau$ is 1-1 so that $\left|{\bigcap}_{t\in K}t^{-1}F_n\right|\le|B\cap F_n|\cdot|K|$. Therefore,
$$
|B\cap F_n|\ge\frac{1}{|K|}\cdot\left|{\bigcap}_{t\in K}t^{-1}F_n\right|\ge\frac{1}{2|K|}\cdot|F_n|
$$
as required.
\end{proof}

Let $\mathscr{F}[T]$ be the collection of F{\o}lner nets in the discrete amenable group $T$. Now for any flow $(T,X)$ and for all $\varepsilon\in\mathscr{U}_X$ we can define the ``Banach lower density'' of $N_T(x,\varepsilon[x])$ in $T$ as follows:
$$
\textmd{BD}_*(T,x,\varepsilon)=\inf_{\{F_n\,|\,n\in D\}\in\mathscr{F}[T]}\liminf_{n\in D}\frac{\left|F_n\cap N_T(x,\varepsilon[x])\right|}{|F_n|}.
$$
Then as a consequence of Theorem~\ref{thm2.6}, we can obtain the following uniform positive recursion frequency for minimal flows with amenable phase groups.

\begin{thm}\label{thm4.4}
Let $(T,X)$ be any w.l.a.p. flow with $T$ a discrete amenable group; then it holds that
${\inf}_{x\in X}\mathrm{BD}_*(T,x,\varepsilon)>0$
for all $\varepsilon\in\mathscr{U}_X$.
\end{thm}

\begin{proof}
Let $\varepsilon\in\mathscr{U}_X$; and then by (3) of Lemma~\ref{lem4}, as in the proof of Theorem~\ref{thm15}, we can find some $K\in\mathscr{K}_T$ such that
$$
T=K^{-1}N_T(x,\varepsilon[x])\quad \textrm{or equivalently}\quad Kt\cap N_T(x,\varepsilon[x])\not=\emptyset\ \forall t\in T
$$
for all $x\in X$. Then by Lemma~\ref{lem4.3} for any F{\o}lner net $\{F_n\,|\,n\in D\}$ in $T$ and $B=N_T(x,\varepsilon[x])$, we have that
$$
\liminf_{n\in D}\frac{\left|F_n\cap B\right|}{|F_n|}\ge\frac{1}{2|K|}.
$$
This thus completes the proof of Theorem~\ref{thm4.4}.
\end{proof}

\begin{thm}\label{thm4.5}
Let $(T,X)$ be any flow with a discrete amenable phase group $T$ and $x\in X$. Then $x$ is an a.p. point of $(T,X)$ if and only if $\mathrm{BD}_*(T,x,\varepsilon)>0$ for all $\varepsilon\in\mathscr{U}_X$.
\end{thm}

\begin{proof}
First, if $x$ is an a.p. point of $(T,X)$, then for all $\varepsilon\in\mathscr{U}_X$, $N_T(x,\varepsilon[x])$ is syndetic in $T$ so that there is some $K\in\mathscr{K}_T$ such that $T=KN_T(x,\varepsilon[x])$. Further by Lemma~\ref{lem4.3}, $\mathrm{BD}_*(T,x,\varepsilon)\ge\frac{1}{2|K|}$.

Conversely, assume that $\mathrm{BD}_*(T,x,\varepsilon)>0$ for all $\varepsilon\in\mathscr{U}_X$. To show that $x$ is a.p. for $(T,X)$, we first note that for any sets $A,B$ in $T$, $|A\vartriangle B|=|At\vartriangle Bt|$ for all $t\in T$.
Indeed, for any $t\in T$,
$$g\in A\vartriangle B\Leftrightarrow g\in A\cup B, g\not\in A\cap B\Leftrightarrow gt\in At\cup Bt, gt\not\in(A\cap B)t=At\cap Bt\Leftrightarrow gt\in At\vartriangle Bt$$
so that $|A\vartriangle B|=|At\vartriangle Bt|$ for all $t\in T$.

Let $\varepsilon\in\mathscr{U}_X$ be any given. We need to show $N_T(x,\varepsilon[x])$ is syndetic in $T$. Indeed, otherwise, $B:=T\setminus N_T(x,\varepsilon[x])$ is ``thick'' in $T$ in the sense that for all $F\in\mathscr{K}_T$ there is some $t\in T$ with $Ft\subseteq B$.

Now let $\{F_n\,|\,n\in D\}$ be any F{\o}lner net in $T$. Then for all $n\in D$ there is $t_n\in T$ such that $F_nt_n\subseteq B$. Since
$$
\frac{|t(F_nt_n)\vartriangle(F_nt_n)|}{|F_nt_n|}=\frac{|tF_n\vartriangle F_n|}{|F_n|}\to 0\quad \forall t\in T,
$$
$\{F_nt_n\,|\,n\in D\}$ is also a F{\o}lner net in $T$. However,
$$
\frac{|N_T(x,\varepsilon[x])\cap F_nt_n|}{|F_nt_n|}=\frac{|\emptyset|}{|F_nt_n|}=0
$$
so that $\textmd{BD}_*(T,x,\varepsilon)=0$, a contradiction. Thus $N_T(x,\varepsilon[x])$ must be syndetic in $T$ and this shows that $x$ is a.p. for $(T,X)$.
\end{proof}

It should be interested to notice that the sufficiency assertion of Theorem~\ref{thm4.5} still holds for all non-discrete amenable group $T$ by the same discussion.
\subsection*{\textbf{Acknowledgments}}
This project was supported by National Natural Science Foundation of China (Grant Nos. 11431012 and 11271183) and PAPD of Jiangsu Higher Education Institutions.



\end{document}